\documentclass[reqno,11pt]{amsart}

\usepackage{amssymb, amsthm}
\usepackage{txfonts, exscale, fullpage}

\newtheorem{lemma}{Lemma}
\newtheorem{theorem}{Theorem}

\newtheorem{corollary}[theorem]{Corollary}

\theoremstyle{remark}
\newtheorem*{acknowledgment}{Acknowledgment} 

\DeclareMathOperator{\re}{Re}

\newcommand{\eps}{\varepsilonup}
\renewcommand{\alpha}{\alphaup}
\renewcommand{\beta}{\betaup}
\renewcommand{\xi}{\xiup}
\renewcommand{\lambda}{\lambdaup}

\begin{document}

\title{On the convergence of some alternating series}
\author{Angel V. Kumchev}
\address{Department of Mathematics, 7800 York Road, Towson University, Towson, MD 21252}
\email{akumchev@towson.edu}

\maketitle

\section{Introduction}
\label{si}

This note is motivated by a question a colleague of the author's often challenges calculus students with: Does the series
\begin{equation}\label{i.1}
  \sum_{n=1}^{\infty} \frac {(-1)^n|\sin n|}{n}
\end{equation}
converge? This series combines features of several series commonly studied in calculus:
\[
  \sum_{n=1}^\infty \frac {(-1)^{n}}{n}, \quad \sum_{n=1}^\infty \frac {|\sin n|}{n} \quad \text{and} \quad \sum_{n=1}^\infty \frac {\sin(nx)}{n}
\]
come to mind. However, unlike these familiar examples, the series~\eqref{i.1} seems to live on the fringes, just beyond the reach of standard convergence tests like the alternating series test or the tests of Abel and Dirichlet. It is therefore quite natural for an infinite series aficionado to study~\eqref{i.1} in hope to find some clever resolution of the question of its convergence. Yet, the author's colleague reports that although he has posed the above question to many calculus students, he has never received an answer. Furthermore, he confessed that he himself had no answer to that question. As it turns out, there is a good reason for that: the question is quite delicate and is intimately connected to deep facts about Diophantine approximation---facts which the typical second-semester calculus student is unlikely to know. 

The series \eqref{i.1} is obtained by perturbation of the moduli of the alternating harmonic series, which is the simplest conditionally convergent alternating series one can imagine. In this note, we study the convergence sets of similar perturbations of a wide class of alternating series. In particular, the convergence of \eqref{i.1} follows from our results and classical work by Mahler \cite{Ma53} on the rational approximations to $\pi$. 

Let $\mathfrak F$ denote the class of continuous, decreasing functions $f : [1, \infty) \to \mathbb R$ such that
\[
  \lim_{x \to \infty} f(x) = 0, \quad \int_1^{\infty} f(x) \, dx = \infty.
\]
Note that if $f \in \mathfrak F$, then $f$ is a positive function and the alternating series $\sum_n (-1)^nf(n)$ is conditionally convergent. Our goal is to describe the convergence set of the related series
\begin{equation}\label{i.2}
  \sum_{n = 1}^{\infty} (-1)^nf(n)|\sin(n\pi\alpha)|.
\end{equation}

It is natural to start one's investigation of~\eqref{i.2} with the case when $\alpha$ is rational, since in that case the sequence $\{ (-1)^n|\sin(n\pi\alpha)| \}_{n=1}^\infty$ is periodic, and one may hope to see some pattern. Indeed, this turns out to be the case, and one discovers the following result.

\begin{theorem}\label{th1}
  Suppose that $f \in \mathfrak F$ and that $\alpha = a/q$, with $a \in \mathbb Z$, $q \in \mathbb Z^+$, and $\gcd(a, q) = 1$. The series~\eqref{i.2} converges if and only if $q$ is odd.
\end{theorem}

When $\alpha$ is irrational, the convergence of \eqref{i.2} depends on the quality of the rational approximations to $\alpha$. Thus, before we can state our results concerning irrational $\alpha$, we need to introduce some terminology. For $\alpha \in \mathbb R$, let $\| \alpha \|$ denote the distance from $\alpha$ to the nearest integer, i.e.,
\[
  \| \alpha \| = \min\big\{ |\alpha - n| : n \in \mathbb Z \big\}.
\]
Given $\alpha \not\in \mathbb Q$, one can construct a unique sequence $\{ a_n/q_n \}_{n = 1}^\infty$ of rational numbers such that $|q_n\alpha - a_n| = \|q_n \alpha \|$ and, for all $n \ge 2$,
\[
  \min\big\{ \| q\alpha \| : 0 < q < q_n \big\} = \|q_{n-1}\alpha \| > \|q_n\alpha\|.
\]
The rational numbers $a_n/q_n$ are called \emph{best rational approximations} to $\alpha$. The reader can find the detailed construction of the sequence $\{ a_n/q_n \}_{n=1}^\infty$ and some of its basic properties in Cassels~\cite[\S I.2]{Ca57}. In particular, it follows easily from the properties listed in~\cite{Ca57} that
\begin{equation}\label{i.3}
  \frac 1{2q_nq_{n+1}} < \left|\alpha - \frac {a_n}{q_n} \right| < \frac 1{q_nq_{n+1}}.
\end{equation}

We can now state our main theorem.

\begin{theorem}\label{th2}
  Suppose that $f \in \mathfrak F$ and $\alpha \not\in \mathbb Q$, and let $\{ a_n/q_n \}_{n = 1}^{\infty}$ be the sequence of best rational approximations to $\alpha$. Let $\mathcal Q_{\alpha}$ be the set of even denominators $q_n$ such that $q_{n+1} \ge 2q_n$. If the series 
  \begin{equation}\label{i.4}
    \sum_{q_n \in \mathcal Q_{\alpha}} \frac {1}{q_n^2} \int_1^{q_{n+1}} f(x) \, dx
  \end{equation}
  converges, then so does the series \eqref{i.2}.
\end{theorem}
 
By combining Theorem \ref{th2} with various facts about Diophantine approximation, we obtain the following corollaries.

\begin{corollary}\label{th3}
  There is a set $D \subset \mathbb R$, with Lebesgue measure zero, such that the series \eqref{i.2} converges for all real $\alpha \notin D$ and all $f \in \mathfrak F$.
\end{corollary}

\begin{corollary}\label{th4}
  Suppose that $f \in \mathfrak F$ and $\alpha$ is an algebraic irrationality. Then the series \eqref{i.2} converges.
\end{corollary}

\begin{corollary}\label{th5}
  The series \eqref{i.1} converges.
\end{corollary}
  
Theorem \ref{th2} provides a sufficient condition for convergence of alternating series of the form~\eqref{i.2}. It is natural to ask how far is this condition from being also necessary. A closer look at the special case $f(x) = x^{-p}$, $0 < p \le 1$, reveals that sometimes the convergence of \eqref{i.4} is, in fact, equivalent to the convergence of \eqref{i.2}. We have the following result.

\begin{theorem}\label{th6}
  Suppose that $\alpha \not\in \mathbb Q$, and let $\{ a_n/q_n \}_{n = 1}^{\infty}$ and $\mathcal Q_{\alpha}$ be as in Theorem~\ref{th2}. When $0 < p \le 1$, the series 
  \begin{equation}\label{i.a} 
    \sum_{n=1}^{\infty} \frac {(-1)^n|\sin(n\pi\alpha)|}{n^p} 
  \end{equation}
  converges if and only if the series   
  \begin{equation}\label{i.b} 
    \sum_{q_n \in \mathcal Q_{\alpha}} \frac {1}{q_n^2} \int_1^{q_{n+1}} x^{-p} \, {dx}
  \end{equation}
  does.
\end{theorem}

In particular, it follows from Theorem \ref{th6} that the divergence set of \eqref{i.2} can be uncountable. Indeed, recalling a classical construction used by Liouville \cite{Li51} to give the first examples of transcendental numbers, we deduce the following corollary.

\begin{corollary}\label{th7}
  There is an uncountable set $L \subset \mathbb R$, dense in $\mathbb R$, such that the series \eqref{i.a} diverges for all $\alpha \in L$ and all $p \in (0,1]$.
\end{corollary}

\section{Some lemmas from calculus}
\label{s1}

In this section, we collect several technical lemmas needed in the proofs of the theorems. We also need to introduce a couple of pieces of notation. Throughout the remainder of the paper, we write $e(\theta) = e^{2\pi i\theta}$. We also use Landau's \emph{big-$O$} notation: if $B > 0$, we write $A = O(B)$ if there exists a constant $c > 0$ such that $|A| \le cB$. In a few places, we will also encounter inequalities like $|A| \le c(\alpha)B$, where $c(\alpha) > 0$ depends solely on a particular fixed parameter $\alpha$. In such situations, it is often convenient to slightly abuse the standard terminology and talk of a ``constant depending only on $\alpha$'' and to write $A = O_{\alpha}(B)$. 

\begin{lemma}\label{l0}
  Suppose that $f \in \mathfrak F$ and $1 \le X < Y$. Then
  \[
    \left| \sum_{X \le n \le Y} (-1)^nf(n) \right| \le f(X).
  \]
\end{lemma}

\begin{proof}
  This follows from the standard proof of the alternating series test. See Bonar and Khoury \cite[Theorem 1.75]{BoKh06}. 
\end{proof}

\begin{lemma}\label{l1}
  Suppose that $f \in \mathfrak F$, that $h, q$ are integers, with $q \ge 1$, and that $1 \le X < Y$. Then
  \[
    \sum_{X < h + kq \le Y} f(h + kq) = \frac 1q \int_X^Y f(x) \, dx + O(f(X)).
  \]
  Here, the summation is over all integers $k$ such that $X < h + kq \le Y$. 
\end{lemma}

\begin{proof}
  Without loss of generality, we may assume that $h=0$. Comparing areas below and above the graph $y = q^{-1}f(x)$, we have
  \[
    \frac 1q \int_{kq}^{(k+1)q} f(x) \, dx < f(kq) < \frac 1q\int_{(k-1)q}^{kq} f(x) \, dx,
  \]
  for any integer $k \ge 2$. Hence,
  \[
    \frac 1q \int_{X_q}^Y f(x) \, dx < \sum_{X < kq \le Y} f(kq) < f(X_q) + \frac 1q\int_{X_q}^Y f(x) \, dx,
  \]
  where $X_q = q\lfloor X/q\rfloor + q$. The lemma follows easily on noting that $X < X_q \le X+q$.
\end{proof}

\begin{lemma}\label{l2}
  Suppose that $q$ and $r$ are integers, with $3 \le q \le r$. Then
  \[
    \sum_{\substack{ q \le k \le r\\ k \equiv q \!\!\!\! \pmod {2q}}} \frac 1{k^2 - 1} < \frac 2{q^2}.
  \]
\end{lemma}

\begin{proof}
  On writing $k = q(2l - 1)$, $l \in \mathbb Z^+$, we can estimate the given sum by
  \begin{align*}
    \sum_{l = 1}^{\infty} \frac 1{q^2(2l-1)^2 - 1} 
    &< \frac 1{q^2 - 1} + \int_1^{\infty} \frac {dx}{q^2(2x-1)^2 - 1} \\
    &= \frac 1{q^2 - 1} + \frac 1{4q} \ln \left( \frac {q+1}{q-1} \right) \\
    &< \frac 1{q^2 - 1} + \frac 1{4q} \, \frac 2{q - 1} < \frac 2{q^2}. 
  \end{align*}
\end{proof}

\begin{lemma}[Partial summation]\label{l3}
  Suppose that $N$ is a positive integer and $\{a_n\}_{n=1}^{\infty}$, $\{b_n\}_{n=1}^{\infty}$ are two sequences of complex numbers. Then 
  \[
    \sum_{n = 1}^{N} a_nb_n = b_{N}\sum_{n = 1}^{N} a_n - \sum_{m = 1}^{N-1} (b_{m+1} - b_{m}) \sum_{n = 1}^{m} a_n.
  \]
\end{lemma}

\begin{proof}
  This is a special case of Bonar and Khoury \cite[Theorem 2.20]{BoKh06}. 
\end{proof}

\begin{lemma}\label{l4}
  For $x \in \mathbb R$,
  \[
    |\sin(\pi x)| = \frac 2{\pi} - \frac 4{\pi} \sum_{k = 1}^{\infty} \frac {\cos(2\pi kx)}{4k^2 - 1} 
    = \frac {-2}{\pi} \sum_{k = -\infty}^{\infty} \frac {e(kx)}{4k^2 - 1}.
  \]
\end{lemma}

\begin{proof}
  The function $|\sin(\pi x)|$ is an even, continuous, $1$-periodic function, so it can be represented by a Fourier cosine-series of the form
  \[
    |\sin(\pi x)| = \frac {a_0}2 + \sum_{k = 1}^{\infty} a_k\cos(2\pi kx),
  \]
  where
  \[
    a_k = 2\int_0^1 \sin(\pi x)\cos(2k\pi x) \, dx.
  \]
  To complete the proof, one simply needs to evaluate the above integral. 
\end{proof}

\begin{lemma}\label{l5}
  Suppose that $N$ is a positive integer and $\alpha \in \mathbb R \setminus \mathbb Z$. Then 
  \begin{equation}\label{1.1}
    \sum_{n = 0}^{N-1} e(\alpha n) = \frac {e(\alpha N) - 1}{e(\alpha) - 1}.
  \end{equation}
  Furthermore, 
  \begin{equation}\label{1.2}
    \left| \sum_{n = 0}^{N-1} e(\alpha n) \right| \le \min \left( N, \frac 1{2\|\alpha\|} \right).
  \end{equation}
\end{lemma}

\begin{proof}
  Identity \eqref{1.1} follows on noting that the sum on the left is a finite geometric series. Estimating the right side of~\eqref{1.1}, we get
  \[
    \left| \sum_{n = 0}^{N-1} e(\alpha n) \right| = \left| \frac {e(\alpha N) - 1}{e(\alpha) - 1} \right| = \left| \frac {\sin(\pi N\alpha)}{\sin(\pi\alpha)} \right| \le \frac 1{2\|\alpha\|},
  \]
  where the last inequality uses the concavity of $|\sin(\pi \alpha)|$ in the range $0 < \alpha < 1$. The other part of inequality \eqref{1.2} is the trivial bound that follows from the triangle inequality.
\end{proof}

\begin{lemma}\label{l6}
  Suppose that $1 \le \nu < \mu \le \infty$ and $p > 0$. Then 
  \begin{equation}\label{1.3}
    \left| \int_{\nu}^{\mu} t^{-p} \cos t \, dt \right| \le 2\nu^{-p}.
  \end{equation}
\end{lemma}

\begin{proof}
  Suppose first that $\mu < \infty$. Partial integration gives
  \[
    \int_{\nu}^{\mu} t^{-p} \cos t \, dt 
    = t^{-p}\sin t \big|_{\nu}^{\mu} + p \int_{\nu}^{\mu} \frac {\sin t}{t^{p+1}} \, dt.
  \]
  Hence,
  \begin{align*}
    \left| \int_{\nu}^{\mu} t^{-p} \cos t \, dt \right| &\le \nu^{-p} + \mu^{-p} + p \int_{\nu}^{\mu} \frac {dt}{t^{p+1}} = 2\nu^{-p}.
  \end{align*}
  The case $\mu = \infty$ of \eqref{1.3} follows by letting $\mu \to \infty$.
\end{proof}

\begin{lemma}\label{l7}
  Suppose that $0 < p < 1$ and $0 < \nu < 1 < \mu$. Then 
  \begin{equation}\label{1.5}  
    \int_{\nu}^{\mu} t^{-p} \cos t \, dt 
    = A_p + O\big( \mu^{-p} \big) + O_p\big( \nu^{1-p} \big),      
  \end{equation}
where 
  \begin{equation}\label{1.6} 
    A_p = \int_0^\infty t^{-p} \cos t \, dt = \Gamma(1-p)\sin(\pi p/2) > \frac {p}{1-p}.
  \end{equation}
\end{lemma}

\begin{proof}
  Inequality \eqref{1.5} follows from Lemma~\ref{l6} and the bound
  \[
    \bigg| \int_{0}^{\nu} t^{-p} \cos t \, dt \bigg| 
    \le \int_0^\nu t^{-p} \, dt = \frac {\nu^{1-p}}{1-p}.
  \]
  The closed-form expression for $A_p$ is a standard Fourier cosine-transform formula. It can be found in many references on Fourier analysis, though its proof is often omitted. The interested reader will find the most natural proof (which uses the theory of contour integration) in the solution of Problem III.151 in P\'olya and Szeg\"o \cite[p.~331]{PoSz72}. Finally, to derive the lower bound for $A_p$, we use the inequalities
  \[
    \sin(\pi p/2) \ge p, \qquad 
    \Gamma(1-p) = \frac {\Gamma(2-p)}{1-p} \ge \frac {\Gamma(1)}{1-p}.
  \]
\end{proof}

\section{Proof of Theorem \ref{th1}}
\label{s2}

We derive the theorem from Cauchy's criterion. Consider the sum
\begin{equation}\label{2.1}
  S(\alpha; M, N) = \sum_{n=N+1}^{N+M} (-1)^nf(n)|\sin(n\pi\alpha)|,
\end{equation}
where $M, N$ are positive integers. We note that $|\sin(\pi an/q)| = |\sin(\pi ah/q)|$ whenever $n \equiv h \pmod q$. Thus, splitting $S(a/q; M, N)$ according to the residue class of $n$ modulo $q$, we have
\begin{align}\label{2.2}
  S(a/q; M, N) &= \sum_{h = 1}^q \sum_{\substack{ n = N + 1\\ n \equiv h \; (\mathrm{mod} \; q)}}^{N+M} (-1)^nf(n)|\sin(\pi an/q)| \notag\\
  &= \sum_{h = 1}^q |\sin(\pi ah/q)| \sum_{\substack{ n = N + 1\\ n \equiv h \; (\mathrm{mod} \; q)}}^{N+M}  (-1)^{n}f(n).
\end{align}
In \eqref{2.2}, we can express $n$ as
\[
  n = N + h + kq, \qquad 0 \le k \le K = \lfloor (M - h)/q \rfloor.
\]
Hence, we can rewrite \eqref{2.2} as
\begin{align}\label{2.3}
  S(a/q; M, N) &= \sum_{h = 1}^q (-1)^{N_h}|\sin(\pi ah/q)| \sum_{k = 0}^{K}(-1)^{kq}f(N_h + kq),
\end{align}
where $N_h = N + h$. We now consider separately the cases of even and odd $q$.\\

\paragraph*{\em Case 1: $q$ odd.} Then $(-1)^{kq} = (-1)^k$, and we have
\begin{align}\label{2.4}
  \sum_{k = 0}^{K} (-1)^{kq}f(N_h + kq) &= \sum_{k = 0}^{K} (-1)^kf(N_h + kq).
\end{align}
By Lemma \ref{l0}, the sum on the right side of \eqref{2.4} is bounded by $f(N)$. Thus, it follows from \eqref{2.3} that
\[
  |S(a/q; M, N)| \le qf(N),
\]
Since $\lim\limits_{x \to \infty} f(x) = 0$, this establishes the convergence case of Theorem~\ref{th1}.\\

\paragraph*{\em Case 2: $q$ even.} Then $(-1)^{kq} = 1$, and we have
\begin{align}\label{2.5}
  \sum_{k = 0}^{K} (-1)^{kq}f(N_h + kq)&= \sum_{k = 0}^{K} f(N_h + kq).
\end{align}
We apply Lemma \ref{l1} to the sum on the right side of \eqref{2.5} and substitute the result into \eqref{2.3} to obtain
\begin{align}\label{2.6}
  S(a/q; M, N) = \frac {I_f}q \sum_{h = 1}^q (-1)^{N_h}|\sin(\pi ah/q)| + O\left(qf(N)\right),
\end{align}
where 
\[
  I_f = I_f(M,N) = \int_N^{N+M} f(x) \, {dx}.
\]
Since $q$ is even and $\gcd(a,q) = 1$, $a$ must be odd. Thus, $(-1)^h = (-1)^{ah}$, and we have
\begin{align}\label{2.7}
  \sum_{h = 1}^q (-1)^h|\sin(\pi ah/q)| = \sum_{h = 1}^q (-1)^{ah}|\sin(\pi ah/q)|.
\end{align}
Note that when $x \in \mathbb Z$, the expression $(-1)^x|\sin(\pi x/q)|$ depends only on the residue class of $x$ modulo $q$. Also, since $\gcd(a,q) = 1$, the numbers $a, 2a, \dots, qa$ form a complete residue system modulo $q$ (see Hardy and Wright \cite[Theorem~56]{HaWr79}). Therefore, the sum on the right side of \eqref{2.7} is a rearrangement of the sum 
\begin{align}\label{2.8}
  \sum_{j = 1}^q (-1)^j|\sin(\pi j/q)| &= \sum_{j = 1}^{q-1} \sin(\pi j(1 + 1/q)).
\end{align}
An appeal to the well-known formula 
\[
  \sum_{j = 1}^n \sin(j\theta) = \frac {\sin(\frac 12n\theta)\sin(\frac 12(n+1)\theta)}{\sin(\frac 12\theta)}
\]
now yields
\begin{align}\label{2.9}
  \sum_{j = 1}^{q-1} \sin(\pi j(1 + 1/q)) &= \frac {\sin((q+1)\frac {\pi}2)\sin(\frac {q\pi}2 - \frac {\pi}{2q})}{\sin(\frac{\pi}2 + \frac{\pi}{2q})} \notag\\
  &= \frac {(-1)^{q+1}\sin({\pi}/{2q})}{-\cos({\pi}/{2q})} = \tan\left( {\pi}/{2q} \right).
\end{align}
Combining \eqref{2.6}--\eqref{2.9}, we conclude that when $q$ is even and $\gcd(a,q) = 1$, 
\begin{equation}\label{2.10}
  S(a/q; M, N) = \frac {(-1)^NI_f\tan(\pi/2q)}q + O\left(qf(N)\right).
\end{equation} 

To establish the divergence case of Theorem \ref{th1}, we need to show that there are choices of $M$ and $N$, with $N \to \infty$, that keep the right side of \eqref{2.10} bounded away from zero. When $N$ is even, \eqref{2.10} yields
\[
  S(a/q; M, N) \ge \frac {\pi}{2q^2} \int_N^{N+M} f(x) \, {dx} + O \left( qf(N) \right). 
\]
Since $\int_N^{\infty} f(x) \, dx$ diverges, this completes the proof of the theorem.

\section{Proof of Theorem \ref{th2}}
\label{s4}

To prove Theorem \ref{th2}, we again estimate the sum $S(\alpha; M, N)$ defined by \eqref{2.1}. It is convenient to assume that $N$ is even---as we may, since
\[
  S(\alpha; M, N) = S(\alpha; M, N+1) + O\left( f(N) \right).
\]

We start by expanding the function $|\sin(n\pi\alpha)|$ in a Fourier series. By Lemma~\ref{l4},
\[
  S(\alpha;M,N) = \frac {2}{\pi} \sum_{k = -\infty}^{\infty} \frac {-1}{4k^2-1} \sum_{n=N+1}^{N+M} (-1)^nf(n)e(\alpha kn).
\]
Using Lemma \ref{l0} to estimate the contribution from $k = 0$ and combining the terms with $k = \pm m$, $m \ge 1$, we obtain
\begin{align}\label{4.1}
  S(\alpha;M,N) &= \frac {4}{\pi} \re\bigg\{ \sum_{k=1}^{\infty} \frac {-1}{4k^2 - 1} \sum_{n = N+1}^{N+M} (-1)^nf(n)e(\alpha kn) \bigg\} + O\left( f(N) \right).
\end{align}
We now estimate the contribution to the right side of \eqref{4.1} from terms with $k > M$. By the triangle inequality and the monotonicity of $f$,
\[
  \left| \sum_{n = N+1}^{N+M} (-1)^nf(n)e(\alpha kn) \right| \le Mf(N),
\]
whence 
\[
  \sum_{k > M} \frac 1{4k^2-1} \left| \sum_{n = N+1}^{N+M} (-1)^nf(n)e(\alpha kn) \right| 
  \le \sum_{k > M} \frac {Mf(N)}{4k^2-1} = \frac {Mf(N)}{4M+2}.
\]
Thus, we deduce from  \eqref{4.1} that
\begin{align}\label{4.2}
  S(\alpha;M,N) &= \frac {4}{\pi} \re\bigg\{ \sum_{k=1}^M \frac {-e(\alpha kN)}{4k^2 - 1} \sum_{n = 1}^{M} (-1)^ng(n)e(\alpha kn) \bigg\} + O\left( f(N) \right),
\end{align}
where $g(x) = f(N+x)$. By Lemma \ref{l3}, we have
\begin{align}\label{4.3}
  \sum_{n = 1}^{M} (-1)^ng(n)e(\beta n)  
  &= g(M) U(\beta; M) - \sum_{m = 1}^{M-1} \Delta g(m)U(\beta; m),
\end{align}
where $\Delta g(m) = g(m+1) - g(m)$ and
\[
  U(\beta; m) = \sum_{n = 1}^m (-1)^ne(\beta n) = \sum_{n = 1}^m e((\beta + 1/2)n).
\]
Substituting \eqref{4.3} into the right side of \eqref{4.2}, we obtain
\begin{equation}\label{4.4}
  S(\alpha; M, N) = \frac {4}{\pi} \re\bigg\{ g(M)V(\alpha;M) - \sum_{m = 1}^{M-1} \Delta g(m) V(\alpha;m) \bigg\} + O(f(N)),
\end{equation} 
where
\begin{align*}
  V(\alpha; m) = \sum_{k = 1}^{M} \frac {-e(\alpha kN)}{4k^2 - 1} \, U(k\alpha; m).
\end{align*}

In order to estimate the right side of \eqref{4.4}, we break the sum $V(\alpha; m)$ into blocks depending on the denominators of the rational approximations to $\alpha$. Let $\{ a_n/q_n \}_{n=1}^\infty$ be the sequence of best rational approximations to $\alpha$. We want to extract a subsequence $\{ r_j \}_{j=1}^{\infty}$ of $\{ q_n \}_{n=1}^\infty$ that satisfies $r_{j+1} \ge 2r_j$ for all $j \ge 1$. For every $n \ge 1$, there is a unique integer $k = k(n) \ge 0$ such that
\begin{equation}\label{4.5}
  1 \le {q_{n+k}}/{q_n} < 2 \le {q_{n+k+1}}/q_n.
\end{equation}
We construct a recursive sequence $\{ n_j \}_{j = 1}^\infty$ by setting 
\[
  n_1 = 1, \qquad n_{j+1} = n_j + k_j + 1 \quad (j \ge 1), 
\]
where $k_j = k(n_j)$ is chosen according to \eqref{4.5} with $n = n_j$. If we set $r_j = q_{n_j}$, the sequence $\{ r_j \}_{j=1}^\infty$ has the desired property. We decompose $V(\alpha; m)$ into blocks $V_j(\alpha; m)$ defined by
\begin{equation}\label{4.6}
  V_j(\alpha; m) = \sum_{k \in \mathcal K_j(M)} \frac {-e(\alpha kN)}{4k^2 - 1} \, U(k\alpha; m),
\end{equation}
where $\mathcal K_j(M)$ is the set of positive integers $k$ subject to $k \le M$ and $r_{j-1} < 4k \le r_j$. Next, we obtain three different estimates for $V_j(\alpha; m)$. Let $q = q(r_j)$ denote the largest denominator of a best rational approximation to $\alpha$ with $q < r_j$. Note that, by the construction of the $r_j$'s, we have $r_{j-1} \le q < 2r_{j-1}$. Our estimates depend on the size and parity of~$q$.

\subsection{Estimation of $V_j(\alpha; m)$ for small $j$}

When $j$ is bounded above by an absolute constant, we appeal to \eqref{1.2} and get
\begin{equation}\label{4.7}
  |V_j(\alpha; x)| \le \sum_{r_{j-1} < 4k \le r_j} \frac {\|k\alpha + 1/2\|^{-1}}{8k^2 - 2} = K_j(\alpha), \quad \text{say.}
\end{equation}

\subsection{Estimation of $V_j(\alpha; m)$ for odd $q$}
\label{s4.2}

Suppose that $q$ is odd and sufficiently large. Let $a/q$ be the best rational approximation to $\alpha$ with denominator $q$. We write $r = r_j$ and $\theta = \alpha - a/q$. When $4k \le r$, by \eqref{i.3} and the choices of $q$ and $r$, we have
\[
  k|\theta| < \frac {k}{qr} \le \frac 1{4q}. 
\]
Since $q$ is odd, we have $2q \nmid (2ak+q)$ and 
\[
  \delta_{a,q}(k) = \left\| \frac {2ak+q}{2q} \right\| \ge \frac 1{2q}.
\]
Hence,
\[
  \left\| k\alpha + \frac 12 \right\| \ge \delta_{a,q}(k) - k|\theta| >  \delta_{a,q}(k) - \frac 1{4q} 
  \ge \frac {\delta_{a,q}(k)}2.
\]
Using \eqref{4.7}, we obtain
\begin{align}\label{4.8}
  \left| V_j(\alpha; m) \right| &\le \sum_{q/2 < 4k \le r} \frac {\delta_{a,q}(k)^{-1}}{4k^2-1}  
  = \sum_{\substack{ 1 \le h \le 2q\\ 2 \nmid h}} \sum_{\substack{q/2 < 4k \le r\\ 2ak + q \equiv h \!\!\!\! \pmod {2q}}} \frac {\delta_{a,q}(k)^{-1}}{4k^2-1} \notag\\
  &= \sum_{\substack{ 1 \le h \le 2q\\ 2 \nmid h}} \left\| \frac h{2q} \right\|^{-1} \sum_{\substack{q/2 < 4k \le r\\ 2ak + q \equiv h \!\!\!\! \pmod {2q}}} \frac {1}{4k^2-1}.
\end{align} 
Note that we have used the inequality $q \le 2r_{j-1}$ observed earlier. Let $b_h$, $1 \le b_h \le q$, be such that $ab_h \equiv \frac 12(h-q) \pmod q$. The sum over $k$ on the right side of \eqref{4.8} is
\begin{align*}
  \sum_{\substack{q/2 < 4k \le r\\ k \equiv b_h \!\!\!\! \pmod q}} \frac {1}{4k^2-1} 
  &\le \frac {1}{(q/4)^2 - 1} + \sum_{l = 1}^{\infty} \frac 1{4(b_h + lq)^2 - 1} \\
  &\le \frac {20}{q^2} + \sum_{l = 1}^{\infty} \frac 1{3q^2l^2} \le \frac {21}{q^2},
\end{align*}
provided that $q \ge 9$. We deduce from this inequality and \eqref{4.8} that 
\begin{align}\label{4.9}
  \left| V_j(\alpha; m) \right| \le \frac {21}{q^2} \sum_{h = 1}^q \left\| \frac {2h-1}{2q} \right\|^{-1} 
  = \frac {42}{q^2} \sum_{h = 1}^{(q-1)/2} \frac {2q}{2h-1} + \frac {42}{q^2}.
\end{align} 
Finally, combining \eqref{4.9} and the inequality
\[
  \sum_{h = 1}^{(q-1)/2} \frac {2}{2h-1} < 2 + \int_1^{(q-1)/2} \frac {2 \, dx}{2x - 1} < \ln q + 2,
\]
we conclude that 
\begin{align}\label{4.10}
  \left| V_j(\alpha; m) \right| &\le c_1q^{-1}\ln q,
\end{align}  
where $c_1 > 0$ is an absolute constant.

\subsection{Estimation of $V_j(\alpha; m)$ for even $q$}
\label{s4.3}

Suppose that $q$ is even, and let $a, r$ and $\theta$ have the same meanings as in \S\ref{s4.2}. Except when $2k \equiv q \pmod {2q}$, we can argue similarly to \S\ref{s4.2}. Indeed, let $V_j'(\alpha; m)$ be the subsum of $V_j(\alpha; m)$ where $2k \not\equiv q \pmod {2q}$. When $2k \not\equiv q \pmod {2q}$, we have
\[
  \delta_{a,q}(k) = \left\| \frac {2ak+q}{2q} \right\| \ge \frac 1{q},
\]
so we can proceed similarly to \eqref{4.8}--\eqref{4.10} to show that
\begin{align}\label{4.11}
  \left| V_j'(\alpha; m) \right| &\le \sum_{\substack {1 \le |h| \le q\\ 2 \mid h}} \left\| \frac h{2q} \right\|^{-1} \sum_{\substack{q/2 < 4k \le r\\ 2ak \equiv h \!\!\!\! \pmod {2q}}} \frac {1}{4k^2-1} \notag\\
  &\le \frac {42}{q^2}\sum_{j = 1}^{q/2} \frac qj < c_2q^{-1}\ln q,
\end{align} 
where $c_2 > 0$ is an absolute constant.

Thus, it remains to estimate the sum
\begin{align}\label{4.12}
  V_j''(\alpha; m) &= \sum_{\substack{k \in \mathcal K_j(M)\\ 2k \equiv q \!\!\!\! \pmod {2q}}} \frac {-e(\alpha kN)}{4k^2-1} \, U(k\alpha; m).
\end{align} 
Recall that $qr < |\theta|^{-1} < 2qr$. When $2k \equiv q \pmod {2q}$, we have
\[
  U(k\alpha; m) = \sum_{n = 1}^m e(\theta kn),
\]
so \eqref{1.2} gives
\[
  |U(k\alpha; m)| \le \min\left( m, (2k|\theta|)^{-1} \right) \le \min \left( m, 2r \right).
\]
Therefore, by \eqref{4.12} and Lemma \ref{l2},
\begin{align}\label{4.13}
  \left| V_j''(\alpha; m) \right| \le \sum_{\substack{ q \le 2l \le r\\ l \equiv q \!\!\!\! \pmod {2q}}} 
  \frac {2\min(m, r)}{l^2-1} \le 4q^{-2}\min(m,r).
\end{align} 

Combining this inequality and \eqref{4.11}, we conclude that 
\begin{align}\label{4.14}
  \left| V_j(\alpha; m) \right| &\le 4q^{-2}\min(m, r) + c_2q^{-1}\ln q.
\end{align}  
Moreover, we note that the first term on the right side of \eqref{4.14} is superfluous when $2q > r$, since in that case the sum $V_j''(\alpha; m)$ is empty.

\subsection{Completion of the proof}
\label{s4.4}

Let $j_0 \ge 2$ be an integer to be chosen later, and set
\[
  K = \sum_{j = 2}^{j_0} K_j(\alpha) = \sum_{4k \le r_{j_0}} \frac {\|k\alpha + 1/2\|^{-1}}{8k^2 - 2}.
\]
We use \eqref{4.7} to estimate the contribution to $V(\alpha;m)$ from subsums $V_j(\alpha; m)$ with $j \le j_0$, and we use \eqref{4.10} and \eqref{4.14} to estimate the contribution from sums $V_j(\alpha; m)$ with $j > j_0$. Let $\mathcal I_{\alpha}(M)$ denote the set of indices $j > j_0$ such that $r_{j-1} \le M$ and $q(r_j)$ is even and satisfies $2q(r_j) \le r_j$. We obtain 
\[
  |V(\alpha; m)| \le K + c_3\sum_{j \ge j_0} \frac {\ln r_j}{r_j} 
  + 4 \sum_{j \in \mathcal I_{\alpha}(M)} \frac {\min(m, r_j)}{q(r_j)^2},
\] 
where $c_3 = \max(c_1, c_2)$. By our choice of the $r_j$'s, we have $r_j \ge 2^j$, so
\begin{equation}
  \sum_{j \ge j_0} \frac {\ln r_j}{r_j} \le \sum_{j = 1}^{\infty} \frac {j\ln 2}{2^j} = 2\ln 2. \label{4.15}
\end{equation}
Hence,
\begin{equation}\label{4.16}
  |V(\alpha; m)| \le  K + c_4 + 8 \sum_{j \in \mathcal I_{\alpha}(M)} \frac {\min(m, s_j)}{q(r_j)^2},
\end{equation}
where $c_4 = 2c_3\ln 2$ and $s_j = \lceil r_j/2 \rceil$. Using \eqref{4.15} to bound the right side of \eqref{4.4}, we get
\begin{equation}\label{4.17}
  |S(\alpha; M, N)| \le \sum_{j \in \mathcal I_{\alpha}(M)} \frac {11\Sigma_j}{q(r_j)^2} + O_{\alpha,j_0}(f(N)),
\end{equation}
where
\[
  \Sigma_j = g(M)\min(M,s_j) - \sum_{m = 1}^{M-1} \Delta g(m)\min(m, s_j).
\]
Let $\chi_s$ denote the characteristic function of the interval $[0, s]$. Since $\min(m, s) = \sum_{n=1}^m \chi_s(n)$ for integer $s$, Lemma \ref{l3} yields
\begin{equation}\label{4.18}
  \Sigma_j = \sum_{n = 1}^M g(n)\chi_{s_j}(n) \le \sum_{1 \le n \le s_j} g(n) \le \int_0^{s_j} g(x) \, dx.
\end{equation}
By the monotonicity of $f$, 
\[
  \int_0^{s_j} g(x) \, dx = \int_1^{s_j+1} f(x + N - 1) \, dx \le \int_1^{r_j} f(x) \, dx ,
\]
so we deduce from \eqref{4.17} and \eqref{4.18} that 
\begin{equation}\label{4.19}
  |S(\alpha; M, N)| \le \sum_{j \in \mathcal I_{\alpha}(M)} \frac {11}{q(r_j)^2}\int_1^{r_j} f(x) \, dx + O_{\alpha,j_0}(f(N)).
\end{equation}

Finally, let us fix an $\eps > 0$. Since the series \eqref{i.4} converges, we can find an index $n_0 = n_0(\eps)$ such that
\[
  \sum_{ \substack{ n = n_0\\ q_n \in \mathcal Q_\alpha}}^{\infty} \frac 1{q_n^2} \int_1^{q_{n+1}} f(x) \, dx < \frac {\eps}{12}.
\]
We choose $j_0$ above to be the least integer $j \ge 2$ such that $r_j \ge q_{n_0}$. Then
\[
  \sum_{j \in \mathcal I_{\alpha}(M)} \frac 1{q(r_j)^2}\int_1^{r_j} f(x) \, dx \le \sum_{ \substack{ n = n_0\\ q_n \in \mathcal Q_\alpha}}^{\infty} \frac 1{q_n^2} \int_1^{q_{n+1}} f(x) \, dx < \frac {\eps}{12}, 
\]
and \eqref{4.19} yields
\[
  |S(\alpha; M, N)| < \frac {11\eps}{12} + O_{\alpha, \eps}(f(N)).
\]
Therefore, we can find an integer $N_0 = N_0(\eps, \alpha, f)$ such that when $N \ge N_0$, one has
\[
  |S(\alpha; M, N)| < \eps.
\]
This establishes the convergence of the series \eqref{i.2}.

\section{Proof of Theorem \ref{th6}}
\label{s5}

We assume that the series \eqref{i.b} diverges and consider the sum $S(\alpha; M, N)$ one last time. We will use \eqref{4.4} to show that $S(\alpha; M, N)$ can approach $\infty$ as $M, N \to \infty$. We retain the notation introduced in the proof of Theorem \ref{th2} and proceed with the estimation of $S(\alpha; M, N)$. 

Let $j_0 \ge 2$ be a fixed integer chosen so that $r_{j_0}$ is sufficiently large, and let $N$ be a large even integer. We restrict the choice of $N$ to integers of the form $r_j - b$, with $j > j_0$ and $b \in \{ 1, 2\}$. Using \eqref{4.4}, \eqref{4.7}, \eqref{4.10}, \eqref{4.11}, and \eqref{4.15}, we obtain the following version of~\eqref{4.17}:
\begin{equation}\label{5.1}
  S(\alpha; M, N) = \sum_{j \in \mathcal I_{\alpha}(M)} S_j(\alpha; M, N) + O_{\alpha}\left( N^{-p} \right),
\end{equation}
where $\mathcal I_\alpha(M)$ is the set of indices defined in \S\ref{s4.4},
\[
  S_j(\alpha; M, N) = \frac {4}{\pi} \re\bigg\{ g(M)V_j''(\alpha;M) - \sum_{m = 1}^{M-1} \Delta g(m) V_j''(\alpha;m) \bigg\},
\] 
and $V_j''(\alpha; m)$ is the sum defined by \eqref{4.12}. Furthermore, by \eqref{4.13} and the choice of $N$, for indices $j$ with $r_{j-1} \le N$, we have
\[
  \left| V_j''(\alpha; M, N) \right| \le 4q^{-2}r_j \le 4q^{-2}(N+2),
\] 
whence
\[
  \left| S_j(\alpha; M, N) \right| \le c_5r_{j-1}^{-2}N,
\]
for some absolute constant $c_5 > 0$. Thus, from \eqref{5.1},
\begin{equation}\label{5.2}
  S(\alpha; M, N) = \sum_{j \in \mathcal I_{\alpha}'(M,N)} S_j(\alpha; M, N) + O_{\alpha}( N ),
\end{equation}
where $\mathcal I_{\alpha}'(M,N)$ is the set of indices $j \in \mathcal I_{\alpha}(M)$ such that $r_{j-1} > N$.

We now proceed to obtain an approximation for $V_j''(\alpha; m)$, which we will then use to estimate the right side of \eqref{5.2}. Let $a, q, r$ and $\theta$ be as in \S\ref{s4.3}. When $2k \equiv q \pmod {2q}$, we have  
\[
  U(k\alpha; m) = \sum_{n = 1}^m e(kn\theta) = \sum_{n = 0}^{m-1} e(kn\theta) + O(1).
\]
Thus, using \eqref{1.1} and the Taylor expansion $e(z) = 1 + 2\pi iz + O( |z|^2 )$, we find that when $2k \equiv q \pmod {2q}$ and $k \le r$,
\[
  U(k\alpha; m) = \frac {e(km\theta) - 1}{1 - e(k\theta)} + O(1) = \frac {e(km\theta) - 1}{-2\pi ik\theta} + O(1).
\]
We substitute this approximation in \eqref{4.12} and use Lemma \ref{l2} to bound the contribution from the error terms. We obtain
\begin{align*}
  V_j''(\alpha; m) &= \sum_{k \in \mathcal L_j(M)} \frac {e(kN\theta/2)}{k^2-1} \,
  \frac {e(km\theta/2) - 1}{\pi ik\theta} + O\left( q^{-2} \right) \notag\\
  &= \sum_{k \in \mathcal L_j(M)} \frac {1}{k^2-1} \int_N^{N+m} e(kt\theta/2) \, dt + O\left( q^{-2} \right),
\end{align*} 
where $\mathcal L_j(M)$ denotes the set of even integers $k$ such that $\frac 12k \in \mathcal K_j(M)$ and $k \equiv q \pmod {2q}$. Hence,
\[
  S_j(\alpha; M, N) = \frac {4}{\pi} \sum_{k \in \mathcal L_j(M)} \frac {\re\Xi_k(\alpha; M, N)}{k^2-1} , 
\]
where
\[
  \Xi_k(\alpha; M, N) = g(M) \int_N^{N+M} e(kt\theta/2) \, dt - \sum_{m = 1}^{M-1} \Delta g(m) \int_N^{N+m} e(kt\theta/2) \, dt.
\]
Recall that here $g(x) = (N+x)^{-p}$. Interchanging the order of summation and integration in $\Xi_k(\alpha; M, N)$, we find that
\begin{align*}
  \Xi_k(\alpha; M, N) &= \int_N^{N+M} \bigg\{ g(M) - \sum_{ t - N \le m \le M - 1} \Delta g(m) \bigg\} e(kt\theta/2) \, dt \\
  &= \int_N^{N+M} g( \lceil t \rceil - N) e(kt\theta/2) \, dt.  
\end{align*}
Since $\left| g( \lceil t \rceil - N) - t^{-p} \right| \le pt^{-p-1}$ by the mean-value theorem, we obtain
\begin{align*}
  \Xi_k(\alpha; M, N) &= \int_N^{N+M} t^{-p} e(kt\theta/2) \, dt + O\left( N^{-p} \right).  
\end{align*}
Hence, after another appeal to Lemma \ref{l2} to estimate the contribution from the error terms, we have
\begin{align}\label{5.3}
  S_j(\alpha; M, N) &= \frac {4}{\pi} \sum_{k \in \mathcal L_j(M)} \frac 1{k^2-1} \int_N^{N+M} t^{-p} \cos(\pi kt\theta) \, dt + O\left( q^{-2}N^{-p} \right) \notag\\
  &= \frac {4}{\pi} \sum_{k \in \mathcal L_j(M)} \frac {(\pi k|\theta|)^{p-1}} {k^2-1} \int_{\nu_k}^{\nu_k + \mu_k} t^{-p} \cos t \, dt + O\left( q^{-2}N^{-p} \right),
\end{align} 
where $\nu_k = \pi k|\theta|N$ and $\mu_k = \pi k|\theta|M$. Summing over $j$, we deduce from \eqref{5.2} and \eqref{5.3} that 
\begin{align}\label{5.4}
  S(\alpha; M, N) = \frac {4}{\pi} \sum_{j \in \mathcal I_{\alpha}'(M,N)} \sum_{k \in \mathcal L_j(M)} 
  \frac {(\pi k|\theta_j|)^{p-1}}{k^2-1} \int_{\nu_k}^{\nu_k + \mu_k} t^{-p} \cos t \, dt + O_{\alpha}\left( N \right),
\end{align}
where $|\theta_j| = q^{-1}\| q\alpha \|$, $q = q(r_j)$.

In order to estimate the right side of \eqref{5.4}, we will impose some restrictions on the choice of $M$. Let $\mathcal Q_{\alpha}'$ be the subset of $\mathcal Q_\alpha$ containing those $q_n$ for which
\[
  \int_1^{q_{n+1}} x^{-p} \, dx > q_n^{1-p/2},
\]
and let $\mathcal Q_{\alpha}'' = \mathcal Q_{\alpha} \setminus \mathcal Q_{\alpha}'$. The contribution to the series \eqref{i.b} from terms with $q_n \in \mathcal Q_\alpha''$ is dominated by the convergent series $\sum_q q^{-1-p/2}$. Thus, the divergence of \eqref{i.b} implies the divergence of the series
\begin{equation}\label{5.5} 
  \sum_{q_n \in \mathcal Q_{\alpha}'} \frac {1}{q_n^2} \int_1^{q_{n+1}} x^{-p} \, {dx}.
\end{equation}
In particular, the set $\mathcal Q_{\alpha}'$ is infinite. We restrict $M$ to the sequence of numbers of the form $\left\lceil q^{1 + p/3} \right\rceil$, with $q \in \mathcal Q_{\alpha}'$. 

Let $J = J(M, N)$ denote the largest index in the set $\mathcal I_{\alpha}'(M,N)$, and set $\mathcal I_{\alpha}''(M,N) = \mathcal I_{\alpha}'(M,N) \setminus \{ J \}$ and $q = q(r_J)$. Using our restriction on the choice of $M$, Lemma~\ref{l2}, and the bound
\[
  \left| \int_{\nu_k}^{\nu_k + \mu_k} t^{-p} \cos t \, dt \right|
  \le \begin{cases}
    (1-p)^{-1}\mu_k^{1 - p} & \text{if } p < 1, \\
    \ln M                   & \text{if } p = 1,
  \end{cases}
\]
we find that the term with $j = J$ in \eqref{5.4} is bounded above by
\[
   c_6 \sum_{\substack{ q \le 2k \le r_J\\ k \equiv q \!\!\!\! \pmod {2q}}} \frac {q^{1-p/2}}{k^2 - 1} 
   \le 2c_6 q^{-1},
\]
where $c_6 = c_6(p) > 0$ is a constant depending only on $p$. Hence,
\begin{align}\label{5.7}
  S(\alpha; M, N) = \frac {4}{\pi} \sum_{j \in \mathcal I_{\alpha}''(M,N)} \sum_{k \in \mathcal L_j(M)} 
  \frac {(\pi k|\theta_j|)^{p-1}}{k^2-1} \int_{\nu_k}^{\nu_k + \mu_k} t^{-p} \cos t \, dt + O_{\alpha,p}\left( N \right).
\end{align}
Since the integrals on the right side of \eqref{5.7} behave somewhat differently when $p = 1$ and when $0 < p < 1$, we now consider these two cases separately.

\subsection{The case $p = 1$}
\label{s5.1}

When $j \in  \mathcal I_{\alpha}''(M,N)$ and $k \in \mathcal L_j(M)$, we have $\mu_k = \pi k|\theta_j|M > 1$. Hence, by Lemma \ref{l6},
\begin{align*}
  \int_{\nu_k}^{\nu_k + \mu_k} t^{-1} \cos t \, dt &= \int_{\nu_k}^1 t^{-1} \cos t \, dt + O(1) \\
  &= \int_{\nu_k}^1 t^{-1} \left( 1 + O \left(t^2\right) \right) \, dt + O(1) \\
  &= -\ln \nu_k + O(1) = \ln r_j + O(\ln(kN)). 
\end{align*}
From this inequality and \eqref{5.7}, we obtain
\begin{align}\label{5.8}
  S(\alpha; M, N) &= \frac {4}{\pi} \sum_{j \in \mathcal I_{\alpha}''(M,N)} \sum_{k \in \mathcal L_j(M)} 
  \frac {\ln r_j}{k^2-1} + O_{\alpha}\left( N \right) \notag\\
  &\ge \frac 4{\pi} \sum_{j \in \mathcal I_{\alpha}''(M,N)} \frac {\ln r_j}{q(r_j)^2-1} + O_\alpha(N) \notag\\
  &\ge \sum_{q_n \in \mathcal Q_{\alpha}'(M,N)} \frac 1{q_n^2} \int_1^{q_{n+1}} \frac {dt}t + O_\alpha(N),
\end{align}
where $\mathcal Q_{\alpha}'(M,N)$ is the set of those $q_n \in \mathcal Q_{\alpha}'$ for which $q_n > N$ and $q_{n+1} \le M$. In view of the divergence of the series \eqref{5.5}, this establishes that
\[
  \limsup_{M \to \infty} S(\alpha; M, N) = \infty.
\]

\subsection{The case $0 < p < 1$}

When $j \in  \mathcal I_{\alpha}''(M,N)$ and $k \in \mathcal L_j(M)$, by Lemma \ref{l7},
\[
  \int_{\nu_k}^{\nu_k + \mu_k} t^{-p} \cos t \, dt = A_p + O_p\big( \nu_k^{1-p} + \mu_k^{-p} \big),
\]
where $A_p$ is the Fourier integral \eqref{1.6}. From this inequality and \eqref{5.7}, we obtain
\begin{align*}
  S(\alpha; M, N) &= \frac {4A_p}{\pi} \sum_{j \in \mathcal I_{\alpha}''(M,N)} \sum_{k \in \mathcal L_j(M)} 
  \frac {(\pi k|\theta_j|)^{p-1}}{k^2-1} + O_{\alpha,p}\left( N + \Delta \right) \\
  &\ge \frac {4A_p}{\pi^{2-p}} \sum_{j \in \mathcal I_{\alpha}''(M,N)} \frac {r_j^{1-p}}{q(r_j)^2} 
  + O_{\alpha,p}\left( N + \Delta \right),
\end{align*}
where
\[
  \Delta = M^{-p} \sum_{j \in \mathcal I_{\alpha}''(M,N)} \sum_{k \in \mathcal L_j(M)} \frac {(k|\theta_j|)^{-1}}{k^2-1}.
\]
By \eqref{i.3}, Lemma \ref{l2}, and the restriction on $M$, 
\begin{align*}
  \Delta \le 4M^{-p} \sum_{j \in \mathcal I_{\alpha}''(M,N)} \frac {r_j}{q(r_j)^2} 
  \le 4M^{-\gamma_p} \sum_{j \in \mathcal I_{\alpha}''(M,N)} \frac {r_j^{1-p}}{q(r_j)^2},
\end{align*}
where $\gamma_p = \frac 14p^2$. Thus, for sufficiently large values of $M$, we obtain
\begin{align}\label{5.9}
  S(\alpha; M, N) &\ge \frac {2A_p}{\pi^{2-p}} \sum_{j \in \mathcal I_{\alpha}''(M,N)} \frac {r_j^{1-p}}{q(r_j)^2} 
  + O_{\alpha,p}\left( N \right), \notag\\
  &\ge A_p' \sum_{q_n \in \mathcal Q_{\alpha}'(M,N)} \frac 1{q_n^2} \int_1^{q_{n+1}} t^{-p} \, dt 
  + O_{\alpha,p}\left( N \right),
\end{align}
where $A_p' = 2\pi^{p-2}(1-p)A_p > 0$ and $\mathcal Q_{\alpha}'(M,N)$ is defined as in \S\ref{s5.1}. Therefore, once again, using \eqref{5.9} and the divergence of the series \eqref{5.5}, we conclude that
\[
  \limsup_{M \to \infty} S(\alpha; M, N) = \infty.
\]
This completes the proof of the theorem.

\section{Proofs of the corollaries}

In this section, we derive the corollaries from Theorems \ref{th2} and \ref{th6}.

\subsection{Proof of Corollary \ref{th3}}

For $\delta > 0$, let $D_{\delta}$ denote the set of real $\alpha$ such the inequality $q_{n+1} \le q_n^{1 + \delta}$ fails for an infinite number of denominators $q_n$ of best rational approximations to $\alpha$. By a classical theorem on Diophantine approximation due to Khinchin \cite{Kh24}, for any fixed $\delta > 0$, the set $D_{\delta}$ has Lebesgue measure zero. Let $D = D_{1/2}$. Then, for $\alpha \in D$ and $f \in \mathfrak F$, we have
\[
  \int_1^{q_{n+1}} f(x) \, dx \le f(1)q_{n+1} \le f(1)q_n^{3/2}
\]
for all but a finite number of $q_n \in \mathcal Q_\alpha$, and the series \eqref{i.4} is dominated by $\sum_{q \in \mathcal Q_{\alpha}} q^{-1/2}$---which converges, because the elements of $\mathcal Q_\alpha$ grow at least exponentially. Therefore, the series \eqref{i.2} converges by Theorem \ref{th2}.

\subsection{Proof of Corollary \ref{th4}}

When $\alpha$ is an algebraic irrationality, by a celebrated result of Roth \cite{Ro55}, the inequality $q_{n+1} \le q_n^{3/2}$ holds for all but a finite number of denominators $q_n$ of best rational approximations to $\alpha$. Thus, we can argue as in the proof of Corollary \ref{th3}.

\subsection{Proof of Corollary \ref{th5}}

We need some information about the the best rational approximations to $1/\pi$. By a classical result of Mahler~\cite{Ma53}, for all $a, q \in \mathbb Z$ with $q \ge 2$, 
\begin{equation}\label{6.1}
  | \pi - a/q | \ge q^{-42}.
\end{equation}
If $a_n/q_n$ is a best rational approximation to $1/\pi$, with $n$ sufficiently large, we deduce from \eqref{i.3} and \eqref{6.1} that
\[
  \frac 1{q_nq_{n+1}} > \left| \frac 1{\pi} - \frac {a_n}{q_n} \right| > \frac 1{10}\left| {\pi} - \frac {q_n}{a_n} \right| \ge \frac 1{10a_n^{42}} > \frac 1{c_7q_n^{42}},
\]
where $c_7 > 0$ is an absolute constant. Thus, the denominators of best rational approximations to $1/\pi$ satisfy $q_{n+1} \le c_7q_n^{41}$. We can now apply Theorem \ref{th2} with $f(x) = x^{-1}$ and $\alpha = 1/\pi$. In this case, the series \eqref{i.4} takes the form
\[
  \sum_{q_n \in \mathcal Q_{1/\pi}} \frac {\ln q_{n+1}}{q_n^2}.
\]
By the discussion in the preceding paragraph, this series is dominated by $\sum_q q^{-2}\ln q$, so the convergence of \eqref{i.1} follows from Theorem \ref{th2}.

\subsection{Proof of Corollary \ref{th7}}

In 1851, Liouville \cite{Li51} considered the series 
\[
  \xi = \sum_{k=1}^{\infty} 10^{-k!}
\]
and proved that its sum is transcendental, thus furnishing the first known example of a transcendental number. 

More generally, we let $a,q$ be integers, with $\gcd(a,q) = 1$, and $\{ d_k \}_{k=1}^\infty$ be any infinite sequence of $1$'s and $3$'s, and we consider the series
\[
  \lambda = \frac aq + \sum_{k=m}^{\infty} d_k10^{-k!} \qquad (m \ge 1).
\]
The partial sums $\lambda_N = a/q + \sum_{k \le N} d_k10^{-k!}$ satisfy the inequality
\begin{equation}\label{6.2}
  \left| \lambda - \lambda_N \right| \le 3 \sum_{k = N+1}^\infty 10^{-k!} < \frac {10/3}{10^{(N+1)!}}.
\end{equation} 
Let $a_n/q_n$ and $a_{n+1}/q_{n+1}$ be the best rational approximations to $\lambda$ for which $q_n \le q10^{N!} < q_{n+1}$. Then, by the construction of $a_n/q_n$, we have
\[
  \left| \lambda_N - \frac {a_n}{q_n} \right| \le \left| \lambda - \frac {a_n}{q_n} \right| + \left| \lambda - \lambda_N \right| \le 2\left| \lambda - \lambda_N \right| < \frac {7}{10^{(N+1)!}}.
\]
We infer from this inequality and \eqref{i.3} that $q_{n+1} \ge c_8 10^{N\cdot N!}$, where $c_8 > 0$ is a constant depending at most on $q$. Furthermore, when $N$ is sufficiently large, inequality \eqref{6.2} is possible only if $\lambda_N = a_n/q_n$. Therefore, for large $N$, the partial sums $\lambda_N$ belong to the sequence of best rational approximations to $\lambda$; clearly, such $\lambda_N$ have even denominators when expressed in lowest terms. This suffices to establish the divergence of the series \eqref{i.b} at $\alpha = \lambda$ for any fixed $p < 1$. Therefore, the series \eqref{i.a} diverges at the numbers $\lambda$ of the above form. Clearly, the set $L$ of all such $\lambda$ is dense in $\mathbb R$, and it is an exercise in elementary set theory to show that $L$ has the cardinality of the continuum.  

When $p = 1$, we can use a similar argument, but we need to modify the above construction of the $\lambda$'s. In this case, we want the denominators $q_n$ to satisfy the inequality $\ln q_{n+1} \ge c_8 q_n^2$. One way to achieve that is to replace the factor $10^{-k!}$ in the definition of $\lambda$ by $10^{-b_k}$, where $\{ b_k \}_{k = 1}^\infty$ is the recursive sequence defined by
\[
  b_1 = 1, \qquad b_{k+1} = 100^{b_k} \quad (k \ge 1). 
\] 
\\

\begin{acknowledgment}
  The author would like to thank Geoffrey Goodson, Alexei Kolesnikov, Pencho Petrushev, and Houshang Sohrab for several conversations, suggestions and comments during the writing of this paper.
\end{acknowledgment}


\begin{thebibliography}{9}
  \bibitem{BoKh06} 
  D.D. Bonar and M.J. Khoury, 
  \emph{Real Infinite Series}, 
  Mathematical Association of America, 2006.
  \bibitem{Ca57} 
  J.W.S. Cassels, 
  \emph{An Introduction to Diophantine Approximation}, 
  Cambridge University Press, 1957.
  \bibitem{HaWr79}
  G.H. Hardy and E.M. Wright,
  \emph{An Introduction to the Theory of Numbers}, 5th edition,
  Oxford University Press, 1979.
  \bibitem{Kh24}
  A.Ya. Khinchin,
  Einige S\"atze \"uber Kettenbr\"uche, mit Anwendungen auf 
  die Theorie der Diophantischen Approximationen,
  \emph{Math. Ann.} {\bf 92} (1924), 115--125.
  \bibitem{Li51}
  J. Liouville,
  Sur des classes tr\`es-\'etendues de quantit\'es dont la valeur n'est ni alg\'ebrique, 
  ni m\^eme r\'eductible \`a des irrationelles alg\'ebriques, 
  \emph{J. Math. Pures Appl.} {\bf 16} (1851), 133--142.
  \bibitem{Ma53} 
  K. Mahler, 
  On the approximation of $\pi$,
  \emph{Indag. Math.} {\bf 15} (1953), 30--42.
  \bibitem{PoSz72}
  G. P\'olya and G. Szeg\"o,
  \emph{Problems and Theorems in Analysis}, vol. I,
  Springer--Verlag, 1972.
  \bibitem{Ro55}
  K.F. Roth,
  Rational approximations to algebraic numbers, 
  \emph{Mathematika} {\bf 2} (1955), 1--20.
\end{thebibliography}
\end{document}